\documentclass[12pt, reqno]{amsart}
\usepackage{amssymb,amsthm,amsmath}
\usepackage[numbers,sort&compress]{natbib}
\usepackage{amssymb,amsmath}
\usepackage{amsfonts}
\usepackage{mathrsfs}
\usepackage{latexsym}
\usepackage{amssymb}
\usepackage{amsthm}
\usepackage{indentfirst}
\date{February 22, 2015}

\let\oldsection\section
\renewcommand\section{\setcounter{equation}{0}\oldsection}

\newtheorem{theorem}{Theorem}[section]
\newtheorem{lemma}{Lemma}[section]
\newtheorem{proposition}{Proposition}[section]

\newtheorem{remark}{Remark}[section]

\allowdisplaybreaks
\begin{document}

\title[Two-dimensional  Boussinesq equations with vertical dissipation]{Global well-posedness of the 2D Boussinesq equations with vertical dissipation}


\author{Jinkai~Li}
\address[Jinkai~Li]{Department of Computer Science and Applied Mathematics, Weizmann Institute of Science, Rehovot 76100, Israel.}
\email{jklimath@gmail.com}

\author{Edriss~S.~Titi}
\address[Edriss~S.~Titi]{
Department of Mathematics, Texas A\&M University, 3368 TAMU, College Station, TX 77843-3368, USA. ALSO, Department of Computer Science and Applied Mathematics, Weizmann Institute of Science, Rehovot 76100, Israel.}
\email{titi@math.tamu.edu and edriss.titi@weizmann.ac.il}



\begin{abstract}
We prove the global well-posedness of the two-dimensional Boussinesq equations with only vertical dissipation. The initial data $(u_0,\theta_0)$ are required to be only in the space $X=\{f\in L^2(\mathbb R^2)\,|\,\partial_xf\in L^2(\mathbb R^2)\}$, and thus our result generalizes that in [C.~Cao, J.~Wu, \emph{Global regularity for the two-dimensional anisotropic Boussinesq equations with vertical dissipation}, Arch. Rational Mech. Anal., \bf208 \rm(2013), 985--1004], where the initial data are assumed to be in $H^2(\mathbb R^2)$. The assumption on the initial data is at the minimal level that is required to guarantee the uniqueness of the solutions. A logarithmic type limiting Sobolev embedding inequality for the $L^\infty(\mathbb R^2)$ norm, in terms of anisotropic Sobolev norms, and a logarithmic type Gronwall inequality are established to obtain the global in time a priori estimates, which guarantee the local solution to be a global one.
\end{abstract}

\maketitle

\allowdisplaybreaks

\section{Introduction}
\label{sec1}
The Boussinesq equations form a fundamental block in many geophysical models, such as those of atmospheric fronts and ocean circulations, see, e.g., Majda \cite{MAJDA1}, Pedlosky \cite{PED} and Vallis \cite{VALLIS}.
The Boussinesq equations in $\mathbb R^2$ read as
\begin{equation}
  \label{gbe}
  \left\{
  \begin{array}{l}
    \partial_tu+(u\cdot\nabla)u+\nabla p-\nu_1\partial_x^2u-\nu_2\partial_y^2u=\theta e_2,\\
    \text{div}\,u=0,\\
    \partial_t\theta+u\cdot\nabla\theta-\kappa_1\partial_x^2\theta -\kappa_2\partial_y^2\theta=0,
  \end{array}
  \right.
\end{equation}
with $e_2=(0,1)$, where the velocity $u=(u^1, u^2)$, the pressure $p$ and the temperature $\theta$ are the unknowns, while the viscous coefficients $\nu_i$ and diffusivity coefficients $\kappa_i$, $i=1,2$, are nonnegative constants.

In recent years there has been a lot of work investigating the Boussinesq system (\ref{gbe}). On the one hand, it is well-known that the Boussinesq system (\ref{gbe}) with full dissipation, i.e. when all viscous and diffusivity coefficients being positive constants, is global well-posed, see, e.g., Cannon--DiBenedetto \cite{CANDIB} and Foias--Manley--Temam \cite{FOIAS} (see also Temam \cite{TEMAM}); however, on the other hand, for the inviscid case, i.e. when all viscosities and diffusivity coefficients are zero, the global well-posedness is still an open question, and actually, as pointed out in Majda--Bertozzi \cite{MAJDA}, the two-dimensional inviscid Boussinesq equations are identical to the incompressible axi-symmetric swirling three-dimensional Euler equations.

As the intermediate cases between the two extreme cases (the full dissipation case and the inviscid case), the Boussinesq equations with partial dissipation, i.e. with partial viscosities or partial diffusivity, have recently been attracting the attention of many mathematicians. The case that with full viscosities, but without any diffusivity was proved to be global well-posed by Hou--Li \cite{HOULI} and Chae \cite{CHAE}, independently. Note that the case that with full diffusivity, but with no viscosities was also considered and proved to be global well-posed in Chae \cite{CHAE}. The initial data in \cite{CHAE,HOULI} are required to have some suitable high regularities. Some generalizations of the results in  \cite{CHAE,HOULI}, toward weakening the assumption on the initial data, were established in Hmidi--Keraani \cite{HMIKER} and Danchin--Paicu \cite{DANPAI}; in particular, no smoothness assumptions on the initial data are required in \cite{DANPAI}. More precisely, assuming the initial data are in $L^2(\mathbb R^2)$ is sufficient for both global existence and uniqueness. Therefore, the two-dimensional Boussinesq system with full viscosities and with or without diffusivity is global well-posed for any $L^2(\mathbb R^2)$ initial data. It should be pointed out that the global well-posedness results also hold for the system defined in bounded domains of $\mathbb R^2$; however, in this case, the initial data are required to have some appropriate smoothness, see, e.g., Hu--Kukavica--Ziane \cite{HUKUZ} and Lai--Pan--Zhao \cite{LAIPANZHAO}.

Note that all the papers \cite{CHAE,HOULI,HMIKER,DANPAI} consider the full viscosities case to guarantee the global well-posedness, i.e. both the horizontal viscosity $\nu_1$ and the vertical viscosity $\nu_2$ are assumed to be positive constants. One can further reduce the assumptions on the viscosities in the Boussinesq system without losing its global well-posedness. As a matter of fact, the works of Danchin--Paicu \cite{DANPAI2} and of Larios--Lunasin--Titi \cite{LARLUNTIT} show that only horizontal viscosity, i.e., $\nu_1 > 0$, and  $\nu_2 =\kappa_1=\kappa_2=0$, is sufficient for the global well-posedness of the two-dimensional Boussinesq equations, as long as the initial data are suitably regular. It is still unclear if  merely a vertical viscosity, i.e.,   $\nu_2 >0$ and $\nu_1 = \kappa_1=\kappa_2=0$, is sufficient for the global well-posedness of the two-dimensional Boussinesq system. However, it has been proven in Cao--Wu \cite{CAOWU13ARMA} that the two-dimensional Boussinesq equations with only vertical dissipation, i.e., with both vertical viscosity and vertical diffusivity ($\nu_1 = \kappa_1 =0$,  $\nu_2 >0$, and  $\kappa_2>0$), indeed admit a unique global classical solution, for any initial data in $H^2(\mathbb R^2)$.

The aim of this paper is to weaken the assumptions of \cite{CAOWU13ARMA}, concerning the initial data,  as possible as one can, without losing the global well-posedness of the system. More precisely, we are going to establish the global well-posedness of the following Boussinesq system in $\mathbb R^2$, under the minimal assumptions on the initial data at the level that are required for the uniqueness of the solutions
\begin{equation}\label{maineq}
  \left\{
  \begin{array}{l}
  \partial_tu+(u\cdot\nabla)u+\nabla p-\nu\partial_y^2u=\theta e_2,\\
  \text{div}\,u=0,\\
  \partial_t\theta+u\cdot\nabla\theta-\kappa\partial_y^2\theta=0,
  \end{array}
  \right.
\end{equation}
where $\nu$ and $\kappa$ are given positive constants. We consider the Cauchy problem to the above system in $\mathbb R^2$ and look for such solutions that vanish at infinity. We complement the system with the  initial condition
\begin{equation}
  \label{ic}
  (u,\theta)|_{t=0}=(u_0,\theta_0).
\end{equation}
Notably, our result and proof are equally valid for the periodic boundary conditions case, or on the two-dimensional sphere.

Throughout this paper, for $r\in[1,\infty]$ and positive integer $m$, we use $L^r(\mathbb R^2)$ and $H^m(\mathbb R^2)$ to denote the standard Lebessgue and Sobolev spaces on $\mathbb R^2$, respectively, while $H^{-1}(\mathbb R^2)$ is the dual space of $H^1(\mathbb R^2)$. For convenience, we always use $\|f\|_r$ to denote the $L^r(\mathbb R^2)$ norm of $f$.

Next, we state our main result.

\begin{theorem}\label{theorem}
  Assume that $(u_0,\theta_0,\partial_xu_0,\partial_x\theta_0)\in L^2(\mathbb R^2)$, and that $T >0$.  Then, there is a unique global solution $(u,\theta)$ to system (\ref{maineq}), subject to the initial condition (\ref{ic}), satisfying
  \begin{eqnarray*}
    &&(u,\theta)\in C([0,T]; L^2(\mathbb R^2))\cap L^2(0,T; H^1(\mathbb R^2)),\quad\\
    &&(\partial_xu,\partial_x\theta)\in L^\infty(0,T; L^2(\mathbb R^2)),\quad (\partial_{xy}^2u,\partial_{xy}^2\theta)\in L^2(0,T; L^2(\mathbb R^2)),\\
    &&(\partial_tu,\partial_t\theta)\in L^2(0,T; H^{-1}(\mathbb R^2)),\quad (\partial_tu,\partial_t\theta)\in L^2(\tau,T; L^2(\mathbb R^2)),\\
    &&(u,\theta)\in L^\infty(\tau,T; H^1(\mathbb R^2)),\quad(\partial_yu,\partial_y\theta)\in L^2(\tau,T; H^1(\mathbb R^2)),
  \end{eqnarray*}
  for any $\tau\in (0,T]$.
\end{theorem}

\begin{remark}
(i) Our result is clear a generalization of that in \cite{CAOWU13ARMA},
where the $H^2(\mathbb R^2)$ smoothness on the initial data is required in \cite{CAOWU13ARMA}.

(ii) As it will be seen later, the regularity that $(\partial_xu,\partial_x\theta)\in L^2(0,T; L^2(\mathbb R^2))$ is required for the uniqueness, while such kind of regularity for (\ref{maineq}) can only be inherited from the initial data, because of the absence of the horizontal dissipation. Thus, our assumption on the initial data that $(\partial_xu_0,\partial_x\theta_0)\in L^2(\mathbb R^2)$ is natural and necessary for our proof.
\end{remark}

To prove Theorem \ref{theorem}, it suffices to prove the local well-posedness and establish some appropriate a priori estimates. The local well-posedness can be proven in a standard way. Thanks to the dissipation in the vertical direction, the local solution $(u,\theta)$ immediately belongs to $H^1(\mathbb R^2)$ away from the initial time. It turns out that the a priori $L^\infty(H^1(\mathbb R^2))$ estimate of $(u,\theta)$, away from the initial time, will be sufficient for extending the local solution to be a global one. Regarding the $H^1$ type estimates on $(u,\theta)$, we can obtain the following energy inequality
\begin{align*}
  \frac{d}{dt}\|(u,\theta)\|_{H^1}^2+\|(\partial_yu,\partial_{y}\theta) \|_{H^1}^2
  \leq 8(\|u^2\|_\infty^2+\|\theta\|_\infty^2+1)\|(u,\theta)\|_{H^1}^2,
\end{align*}
which indicates that we have to control $\|u^2\|_\infty^2$ in order to obtain the desired a priori $L^\infty(0,T; H^1(\mathbb R^2))$ estimate. To this end, we first establish a logarithmic type embedding inequality for anisotropic Sobolev spaces (see Lemma \ref{lemlogsob} below), which in particular implies that
\begin{align*}
  \|u^2\|_{\infty}^2\leq& C\max\left\{1,\sup_{r\geq2} \frac{\|u^2\|_{r}^2 }{ r\log r }\right\}\log(e^3+\|u\|_{H^1}+\|\partial_yu\|_{H^1}) \\ &\times\log\log(e^3+\|u\|_{H^1}+\|\partial_yu\|_{H^1}).
\end{align*}
Then recalling the previous $H^1$ energy inequality, and thanks to the fact that $\|u^2\|_r$ grows no faster than $\sqrt {r\log r}$, which is established in \cite{CAOWU13ARMA}, we end up with the following kind inequality
$$
A'(t)+B(t)\leq CA(t)\log B(t)\log\log B(t),
$$
where $A$ and $B$ are quantities involving $\|(u,\theta)\|_{H^1}^2$ and
$\|(\partial_yu,\partial_y\theta)\|_{H^1}^2$, respectively. Observing that the above inequality implies the boundedness, up to any finite time $T$, of the quantity $2A(t)+\int_0^tB(s)ds$ (see Lemma \ref{lemloggro}, below), and thus one gets the desired a priori $H^1$ estimate of $(u,\theta)$ and further establish a global solution.

Since the specific values of the positive coefficients $\nu$ and $\kappa$ play no roles in the argument of this paper, and thus, for simplicity, we suppose that
$$
\nu=\kappa=1.
$$
Throughout this paper, we always use $C$ to denote a generic positive constant which may vary from line to line.

The rest of this paper is organized as follows: in the next section, section \ref{sec2}, we state some preliminary results which will be used in the following sections; in section \ref{sec3}, we prove the global well-posedness of the Boussinesq system (\ref{maineq}) with $H^1(\mathbb R^2)$ initial data, based on which we can prove Theorem \ref{theorem} in the last section, section \ref{sec4}.

\section{Preliminaries}
\label{sec2}
In this section, we state some preliminary results, including a logarithmic type embedding inequalities for the anisotropic Sobolev spaces and a logarithmic type Gronwall inequality.

We first recall the following lemma.

\begin{lemma}[See Lemma 2.2 of \cite{CAOWU13ARMA}]
  \label{lemhold}
  Let $q\in[2,\infty)$, and assume that $f, g, \partial_yg, \partial_xh\in L^2(\mathbb R^2)$ and $h\in L^{2(q-1)}(\mathbb R^2)$. Then
  $$
  \int_{\mathbb R^2}|fgh|dxdy\leq C\|f\|_2\|g\|_2^{1-\frac1q}\|\partial_yg\|_2^{\frac1q}\|h\|_{2(q-1)}^{1-\frac1q}
  \|\partial_xh\|_2^{\frac1q},
  $$
  where $C$ is a constant depending only on $q$, and in particular, we have
  $$
  \int_{\mathbb R^2}|fgh|dxdy\leq C\|f\|_2\|g\|_2^{\frac12}\|\partial_yg\|_2^{\frac12}\|h\|_2^{\frac12}\|\partial_x h\|_2^{\frac12},
  $$
  for an absolute positive constant $C$.
\end{lemma}

Next, we state and prove a logarithmic type limiting Sobolev embedding inequalities for the anisotropic Sobolev spaces in $L^\infty (\mathbb{R}^2)$,  which generalizes that in Cao--Li--Titi \cite{CAOLITITI}, where the relevant inequality for the standard Sobolev spaces was established. Notably, this inequality also generalizes the classical Brezis--Gallouate--Wainger inequality \cite{Brezis_Gallouet_1980,Brezis_Wainger_1980} (see also \cite{CAOFARHATTITI,CAOLITITI-CPAM,CAOWU13ARMA}, and references therein for other similar inequalities).

\begin{lemma}[Logarithmic limiting Sobolev embedding]\label{lemlogsob}
Denote $\textbf{p}=(p_1, p_2,\cdots, p_N)$, with $p_i\in(1,\infty)$, and suppose that $\sum_{i=1}^N\frac{1}{p_i}<1.$
Then, we have
\begin{align*}
  \|F\|_{L^\infty(\mathbb R^N)}\leq& C_{N,\textbf{p},\lambda}\max\left\{1,\sup_{r\geq2} \frac{\|F\|_{L^r(\mathbb R^N)}}{(r\log r)^\lambda}\right\}[\log \mathcal N_{\textbf{p}}(F) \log\log \mathcal N_{\textbf{p}}(F) ]^\lambda,
\end{align*}
for any $\lambda>0$ and for any function $F$ such that all the quantities on the right-hand side are well-defined and finite, where
$$
\mathcal N_{\textbf{p}}(F)=\sum_{i=1}^N(\|F\|_{L^{p_i}(\mathbb R^N)}+\|\partial_iF\|_{L^{p_i}(\mathbb R^N)})+e^3.
$$
\end{lemma}

\begin{proof}
Without loss of generality, we can suppose that
$|F(0)|=\|F\|_\infty$. Choose a function $\phi\in C_0^\infty(B_2)$, with $0\leq\phi\leq1$, and $\phi\equiv1$ on $B_{1}$. Setting $f=F\phi$, we have (see (5.3) in \cite{CAOLITITI})
\begin{equation}\label{2.5}
|f(0)|^q\leq C_{N,\textbf{p}}q\sum_{i=1}^N \|f\|_{(q-1)\kappa_i}^{q-1}\|\partial_if\|_{p_i},
\end{equation}
for any $q\geq3$, where
$$
\kappa_i=\frac{p_i\left(1+\sum_{j=1}^N\alpha_j\right)}{1-\sum_{j=1}^N\alpha_j}, \quad \alpha_i=\frac{1}{p_i}.
$$
It should be mentioned that the proof of (\ref{2.5}) was given for the case $N\geq3$ in \cite{CAOLITITI}; however, as pointed out there that the same inequality also holds for the two-dimensional case by a similar argument (see also Cao--Fahart--Titi \cite{CAOFARHATTITI} for a similar result in $2D$).

One can easily verify that
\begin{eqnarray*}
  &&\|F\|_\infty=|F(0)|=|f(0)|,\quad \|f\|_{(q-1)\kappa_i}\leq\|F\|_{(q-1)\kappa_i},\\
  &&\|\partial_if\|_{p_i}\leq C(\|F\|_{p_i}+\|\partial_iF\|_{p_i})\leq C\mathcal{N}_{\textbf{p}}(F).
\end{eqnarray*}
With the aid of the above inequalities, and noticing that $q^{\frac{1}{q}}\leq C$ and $(q-1)\kappa_i\geq2$, for any $q\in[3,\infty)$, we deduce from (\ref{2.5}) that
\begin{align*}
  \|F\|_\infty\leq&C_{N,\textbf{p}}\sum_{j=1}^N\|F\|_{(q-1)\kappa_i}^{1-\frac{1}{q}} \mathcal{N}_{\textbf{p}}(F)^{\frac{1}{q}}\\
  =&C_{N,\textbf{p}}\sum_{i=1}^N\left\{\frac{\|F\|_{(q-1) \kappa_i}^{1-\frac{1}{q}}} {[(q-1)\kappa_i\log((q-1)\kappa_i)]^{\lambda(1-\frac{1}{q})}}\right. \\
  &\quad\times[(q-1)\kappa_i\log((q-1)\kappa_i)] ^{\lambda\left(1-\frac{1}{q}\right)}\bigg\} \mathcal{N}_{\textbf{p}}(F)^{\frac{1}{q}}\\
  \leq&C_{N,\textbf{p},\lambda}\sum_{i=1}^N \left\{\frac{\|F\|_{(q-1)\kappa_i}}{[(q-1)\kappa_i\log((q-1)\kappa_i)]
  ^\lambda}
  \right\}^{1-\frac{1}{q}}(q\log q)^{\lambda}\mathcal{N}_{\textbf{p}}(F)^{\frac{1}{q}}\\
  \leq&C_{N,\textbf{p},\lambda}\max\left\{1,\sup_{r\geq2} \frac{\|F\|_r}{(r\log r)^\lambda}\right\}(q\log q)^\lambda\mathcal{N}_{\textbf{p}}(F)^{\frac{1}{q}}.
\end{align*}
Therefore
$$
  \|F\|_\infty\leq C_{N,\textbf{p},\lambda}\max\left\{1,\sup_{r\geq2} \frac{\|F\|_r}{(r\log r)^\lambda}\right\}
   (q\log q)^\lambda\mathcal{N}_{\textbf{p}}(F)^{\frac{1}{q}} ,
$$
for any $q\geq 3$. Noticing that $\log\mathcal{N}_{\textbf{p}}(F)\geq3$, one can choose $q=\log\mathcal{N}_{\textbf{p}}(F)$ in the above inequality and end up with
\begin{align*}
  \|F\|_\infty\leq&C_{N,\textbf{p},\lambda}\max\left\{1,\sup_{r\geq2} \frac{\|F\|_r}{(r\log r)^\lambda}\right\}[\log\mathcal{N}_{\textbf{p}}(F) \log\log\mathcal{N}_{\textbf{p}}(F)]^\lambda,
\end{align*}
for any $\lambda>0$, which proves the conclusion.
\end{proof}

Finally, we prove a logarithmic Gronwall inequality stated in the following lemma. This lemma will be applied later to establish the global in time a priori estimates.

\begin{lemma}[Logarithmic Gronwall inequality]
\label{lemloggro}
 Let $T\in(0,\infty)$ be given.  Let $A$ and $B$ be defined and integrable functions on  $(0,T)$, with $A,B\geq e$,  such that $A$ is absolutely continuous on $(0,T)$ and continuous on $[0,T)$. Suppose that
\begin{equation}\label{2.1}
A'(t)+B(t)\leq KA(t)\log B(t)\log\log B(t),
\end{equation}
for $t\in(0,T)$, where $K\geq1$ is a constant. Then
$$
2A(t)+\int_0^tB(t)\leq 512K^2Q^2(t)+2A(0),
$$
for any $t\in[0,T)$, where
$$
Q(t)=\exp\left\{\exp\left\{(\log\log A(0)+260K^2t)e^{Kt}\right\}\right\}.
$$
\end{lemma}

\begin{proof}
Dividing both sides of inequality (\ref{2.1}) by $A$, and defining
$$
A_1=\log A,\quad B_1=\frac{B}{A},
$$
one has
\begin{equation}
A_1'+B_1\leq K(A_1+\log B_1)\log(A_1+\log B_1).\label{2.2}
\end{equation}
One can easily verify that
\begin{equation}
\log z\leq z,\quad\log z\leq 4z^{1/4},\quad\mbox{for }z\in(0,\infty).\label{2.4}
\end{equation}
Thus, we can deduce from (\ref{2.2}), (\ref{2.4}) and the Young inequality that
\begin{align*}
  A_1'+B_1\leq&K(A_1+\log B_1)\log(A_1+B_1)\\
  =&KA_1\log(A_1+B_1)+K\log B_1\log(A_1+B_1)\\
  \leq&KA_1\log(A_1+B_1)+16KB_1^{1/4}(A_1+B_1)^{1/4}\\
  \leq&KA_1\log(A_1+B_1)+16K(A_1+B_1)^{1/2}\\
  \leq&\frac12(A_1+B_1)+\frac12(16K)^2+KA_1\log(A_1+B_1),
\end{align*}
from which, one gets
$$
A_1'+\frac12(A_1+B_1)\leq A_1+(16K)^2+KA_1\log(A_1+B_1).
$$
Dividing both sides of the above inequality by $A_1$, noticing that $A_1\geq1$, and denoting
$$
A_2=\log A_1, \quad B_2=\frac{A_1+B_1}{2A_1},
$$
one obtains
\begin{equation}
  \label{2.3}
  A_2'+B_2\leq1+K+(16K)^2+K(A_2+\log B_2).
\end{equation}
One can easily verify that
$$
\log z\leq 2z^{1/2},\quad\mbox{ for }z\in(0,\infty).
$$
Thus, recalling that $K\geq1$, it follows from (\ref{2.3}) and the Young inequality that
\begin{align*}
  A_2'+B_2\leq&1+K+(16K)^2+KA_2+2KB_2^{1/2}\\
  \leq&\frac12B_2+1+K+2K^2+(16K)^2+KA_2
  \leq \frac12B_2+260K^2+KA_2,
\end{align*}
which implies
$$
A_2'+\frac12B_2\leq260K^2+KA_2.
$$
Applying the Gronwall inequality to the above inequality yields
$$
A_2(t)\leq e^{Kt}(A_2(0)+260K^2t)=e^{Kt}(\log\log A(0)+260K^2t),
$$
for any $t\in[0,T)$, and therefore, recalling the definition of $A_2$, we have
$$
A(t)=e^{e^{A_2(t)}}\leq \exp\left\{\exp\left\{(\log\log A(0)+260K^2t)e^{Kt}\right\}\right\}=:Q(t).
$$
On account of the above estimate, and recalling (\ref{2.4}), it follows from (\ref{2.1}) and the Young inequality that
\begin{align*}
  A'+B\leq&KA\log B\log\log B\leq KA(\log B)^2\\
  \leq&KA(4B^{1/4})^2=16KAB^{1/2}\leq 16KQ(t)B^{1/2}\\
  \leq&\frac12B+(16K)^2Q^2(t)=\frac12B+256K^2Q^2(t).
\end{align*}
Integrating the above inequality over $(0,t)$ yields the conclusion.
\end{proof}

\section{Global well-posedness : $H^1$ initial data}
\label{sec3}
In this section, we prove the global well-posedness of solutions to system (\ref{maineq}) with $H^1$ initial data. This result will be used to prove the global well-posedness with lower regular initial data in the next section.

We first recall the following global well-posedness result with $H^2$ initial data.

\begin{proposition}[See Theorem 1.1 in \cite{CAOWU13ARMA}]
  \label{lemglo}
 For any initial data $(u_0,\theta_0)\in H^2(\mathbb R^2)$, there is a unique global solution $(u,\theta)$ to system (\ref{maineq}), with initial data $(u_0,\theta_0)$, satisfying
 $$
 (u,\theta)\in C([0,\infty); H^2(\mathbb R^2)).
 $$
\end{proposition}

The solution $(u,\theta)$ to system (\ref{maineq}), with initial data $(u_0,\theta_0)\in H^2(\mathbb R^2)$, satisfies some basic energy estimates stated in the following proposition.

\begin{proposition}[See Lemma 3.3 in \cite{CAOWU13ARMA}]\label{lembas}
Let $(u,\theta)$ be a classical solution to system (\ref{maineq}), with initial data $(u_0,\theta_0)\in H^2(\mathbb R^2)$, then one has
\begin{eqnarray*}
  &&\sup_{0\leq s\leq t}\|u(s)\|_2^2+2\int_0^t\|\partial_yu\|_2^2ds \leq (\|u_0\|_2+t\|\theta_0\|_2)^2,\\
  &&\sup_{0\leq s\leq t}\|\theta(s)\|_q^q+q(q-1)\int_0^t\||\theta|^{\frac{q}{2}-1}\partial_y\theta \|_2^2ds\leq\|\theta_0\|_q^q,\quad\mbox{for every } q\in(1,\infty).
\end{eqnarray*}
\end{proposition}

It turns out that the $L^r$ norms of the second component of the velocity grow no faster than $\sqrt{r\log r}$, more precisely, the following proposition holds.

\begin{proposition}[See Proposition 4.1 in \cite{CAOWU13ARMA}]
\label{lemgrowth}
Let $(u,\theta)$ be a classical solution to system (\ref{maineq}), with initial data $(u_0,\theta_0)\in H^2(\mathbb R^2)$, then one has
\begin{equation}\label{1}
\sup_{r\geq 2}\frac{\|u^2(t)\|_{2r}^2}{ {r\log r}}\leq\sup_{r\geq 2}
\frac{\|u^2_0\|_{2r}^2}{ {r\log r}}+m(t),
\end{equation}
where $m(t)$ is an explicit nondecreasing continuous function of $t\in[0,\infty)$ that depends continuously on the initial norm $\|(u_0,\theta_0)\|_{H^1}$.
\end{proposition}

\begin{remark}
In the original Proposition 4.1 in \cite{CAOWU13ARMA}, the function $m(t)$ is stated to be locally integrable on $[0,\infty)$ and depends on $\|(u_0,\theta_0)\|_{H^2}$; however, one can check the proof there to find that it is actually continuous in time and depends continuously on the initial norm $\|(u_0,\theta_0)\|_{H^1}$.
\end{remark}

As a corollary of Proposition \ref{lemgrowth}, and using the following inequality (see, e.g., inequality (3) on page 206 of Lieb--Loss \cite{LIEBLOSS}, which is a sharper version of the following inequality)
\begin{equation}
\|f\|_{L^q(\mathbb R^2)}^2\leq Cq\|f\|_{H^1(\mathbb R^2)}^2,\label{sob2d}
\end{equation}
we then have the following estimate
\begin{equation}
  \label{growth}
  \sup_{r\geq 2}\frac{\|u^2(t)\|_{2r}^2}{ {r\log r}}\leq
  C\|u_0\|_{H^1}^2+m(t),
\end{equation}
where $m(t)$ is the same function as in (\ref{1}).

Based on the estimate on the growth of the $L^r$ norms of $u^2$, and using the logarithmic type Sobolev embedding inequality (Lemma \ref{lemlogsob}) and the logarithmic type Gronwall inequality (Lemma \ref{lemloggro}), we can successfully establish the a priori $H^1$ estimates stated in the following proposition, which is the key estimate of this paper.

\begin{proposition}[$H^1$ estimate]\label{lemh1est}
    Let $(u,\theta)$ be a solution to system (\ref{maineq}), with initial data $(u_0,\theta_0)\in H^2(\mathbb R^2)$. Then, we have
    \begin{align*}
      \sup_{0\leq s\leq t}(\|u(s)\|_{H^1}^2+\|\theta(s)\|_{H^1}^2)+ \int_0^t (\|\partial_yu\|_{H^1}^2+\|\partial_y\theta\|_{H^1}^2)ds\leq S_1(t),
    \end{align*}
    where $S_1$ is an explicit nondecreasing continuous function of $t\in[0,\infty)$, depending continuously on the initial norm $\|(u_0,\theta_0)\|_{H^1}$.
\end{proposition}

\begin{proof}
Take arbitrary $T\in(0,\infty)$, and define the functions
$$
A(t)=\|u(t)\|_{H^1}^2+\|\theta(t)\|_{H^1}^2+e^3,\quad B(t)=\|\partial_yu(t)\|_{H^1}^2+\|\partial_y\theta(t)\|_{H^1}^2+e^3,
$$
for any $t\in[0,T]$.

Recall that, for any divergence free vector function $u\in H^2(\mathbb R^2)$, one has
$$
\int_{\mathbb R^2}(u\cdot\nabla u)\cdot\Delta udxdy=0.
$$
Multiplying equations $(\ref{maineq})_1$ and $(\ref{maineq})_3$ by $u-\Delta u$ and $\theta-\Delta\theta$, respectively, summing the resultants up, and integrating over $\mathbb R^2$, then it follows from integration by parts, Lemma \ref{lemhold} and the Young inequality that
\begin{align*}
  &\frac{1}{2}\frac{d}{dt} (\|u\|_{H^1}^2+\|\theta\|_{H^1}^2) +(\|\partial_yu\|_{H^1}^2+\|\partial_{y}\theta\|_{H^1}^2)\\
  =&\int_{\mathbb R^2}[\theta(u^2-\Delta u^2)-(\partial_xu\cdot\nabla\theta\partial_x\theta+ \partial_yu\cdot\nabla\theta\partial_y\theta)]dxdy\nonumber\\
  =&\int_{\mathbb R^2}(\theta u^2+\partial_x\theta\partial_xu^2-\theta\partial_y^2u^2+\partial_yu^2\partial_x\theta\partial_x \theta\\
  &-\partial_xu^2\partial_y\theta\partial_x \theta-\partial_yu^1\partial_x\theta\partial_y\theta-\partial_yu^2 \partial_y\theta\partial_y\theta) dxdy\nonumber\\
  =&\int_{\mathbb R^2}[\theta u^2+\partial_x\theta\partial_xu^2-\theta\partial_y^2u^2-2u^2\partial_x \theta \partial_{xy}^2\theta+\theta(\partial_{xy}^2u^2\partial_x \theta\nonumber\\
  &+\partial_x u^2\partial_{xy}^2\theta+\partial_y^2u^1\partial_x\theta +\partial_yu^1 \partial_{xy}^2\theta)+2u^2\partial_y\theta \partial_y^2\theta]dxdy\nonumber\\
  \leq&\|\theta\|_2\|u\|_2+\|\nabla\theta\|_2\|\nabla u\|_2+\|\theta\|_2\|\nabla\partial_yu\|_2+2[\|u^2\|_\infty \|\nabla\theta\|_2\|\nabla\partial_y\theta\|_2\\
  &+\|\theta\|_\infty(\|\nabla\theta\|_2\|\nabla\partial_yu\|_2+\|\nabla u\|_2\|\nabla\partial_y\theta\|_2)]\nonumber\\
  \leq&\frac12 \|(\nabla\partial_{y}u,\nabla\partial_{y}\theta)\|_2^2 +4(\|u^2\|_\infty^2+\|\theta\|_\infty^2+1)\|(u,\theta)\|_{H^1}^2.
\end{align*}
Thus, we have
$$
A'(t)+B(t)\leq 8(1+\|\theta\|_\infty^2+\|u^2\|_\infty^2)A(t),
$$
for any $t\in(0,T]$.

By Lemma \ref{lemlogsob} and Proposition \ref{lembas}, and recalling (\ref{sob2d}) and (\ref{growth}), we have
\begin{align*}
  \|\theta\|_\infty^2\leq&C\max\left\{1,\sup_{r\geq 2}\frac{\|\theta\|_r^2}{r\log r}\right\}\log\mathcal N(\theta)\log\log \mathcal N(\theta)\\
  \leq&C\max\left\{1,\sup_{r\geq 2}\frac{\|\theta_0\|_r^2}{r\log r}\right\}\log\mathcal N(\theta)\log\log \mathcal N(\theta)\\
  \leq&C(1+\|\theta_0\|_{H^1}^2)\log\mathcal N(\theta)\log\log \mathcal N(\theta),
\end{align*}
and
\begin{align*}
  \|u^2\|_\infty^2\leq&C\max\left\{1,\sup_{r\geq2}\frac{\|u^2\|_r^2}{r \log r}\right\}\log\mathcal N(u^2)\log\log\mathcal N(u^2)\\
  \leq&C(1+\|u_0\|_{H^1}^2+m(t))\log\mathcal N(u^2)\log\log\mathcal N(u^2),
\end{align*}
where
\begin{eqnarray*}
  &&\mathcal N(\theta)=\|\theta\|_2+\|\partial_x\theta\|_2+\|\theta\|_4+\|\partial_y\theta \|_4+e^3,\\
  &&\mathcal N(u^2)=\|u^2\|_2+\|\partial_xu^2\|_2+\|u^2\|_4 +\|\partial_yu^2\|_4+e^3.
\end{eqnarray*}

By the Ladyzhenskaya interpolation inequality for $L^4(\mathbb{R}^2)$ (cf. \cite{Constantin_Foias_1988,TEMAM}), one has
$$
\mathcal N(\theta)+\mathcal N(u^2)\leq C(\|(u,\theta)\|_{H^1}+\|(\partial_yu,\partial_y\theta)\|_{H^1}+e^3)\leq C(A+B).
$$
Therefore, recalling that $m(t)$ is a nondecreasing function, we have
\begin{align*}
  A'(t)+A(t)+B(t)\leq C(1+\|(u_0,\theta_0)\|_{H^1}^2+m(T))A\log(A+B)\log\log(A+B),
\end{align*}
for any $t\in(0,T]$.
Applying Lemma \ref{lemloggro} to the above inequality, one has
\begin{align*}
  2\sup_{0\leq t\leq T}A(t)+\int_0^TB(t)dt\leq S_1(T),
\end{align*}
for an explicit continuous increasing function $S_1$, which depends continuously on the initial norm $\|(u_0,\theta_0)\|_{H^1}$. This completes the proof.
\end{proof}

We also have the a priori estimates on the time derivatives of the solutions, that is we have the following:

\begin{proposition}\label{lemtest}
Let $(u,\theta)$ be a  solution to system (\ref{maineq}), with initial data $(u_0,\theta_0)\in H^2(\mathbb R^2)$. Then, we have
    \begin{align*}
      \int_0^t (\|\partial_tu\|_{2}^2+\|\partial_t\theta\|_{2}^2)ds\leq S_2(t),
    \end{align*}
    where $S_2$ is an explicit nondecreasing continuous function of $t\in[0,\infty)$, depending continuously on the initial norm $\|(u_0,\theta_0)\|_{H^1}$.
\end{proposition}

\begin{proof}
  Multiplying equations $(\ref{maineq})_1$ and $(\ref{maineq})_3$ by $\partial_tu$ and $\partial_t\theta$, respectively, summing the resultants up, and integrating over $\mathbb R^2$, then it follows from Lemma \ref{lemhold} that
  \begin{align*}
&\frac12\frac{d}{dt}(\|\partial_yu\|_2^2+\|\partial_y\theta\|_2^2) +(\|\partial_tu\|_2^2+\|\partial_t\theta\|_2^2)\\
    =&\int_{\mathbb R^2}(\theta\partial_tu^2-(u\cdot\nabla)u\cdot\partial_tu-u\cdot\nabla\theta \partial_t\theta)dxdy\\
    \leq&\|\theta\|_2\|\partial_tu\|_2+C\|u\|_2^{\frac12} \|\partial_xu\|_2^{\frac12}\|\nabla u\|_2^{\frac12}\|\nabla\partial_y u\|_2^{\frac12}\|\partial_tu\|_2\\
    &+C\|u\|_2^{\frac12}\|\partial_xu\|_2^{\frac12}\|\nabla\theta\|_2^{\frac12} \|\nabla\partial_y\theta\|_2^{\frac12}\|\partial_t\theta\|_2\\
    \leq&\frac12(\|\partial_tu\|_2^2+\|\partial_t\theta\|_2^2)+ C(\|\theta\|_2^2+\|u\|_2^2\|\partial_xu\|_2^2\\
    &+\|\nabla u\|_2^2\|\nabla\partial_yu\|_2^2 +\|\nabla\theta\|_2^2\|\nabla\partial_y\theta\|_2^2),
  \end{align*}
which, integrating with respect to $t$ and using Proposition \ref{lemh1est}, yields the conclusion.
\end{proof}

With the a priori estimates (Propositions \ref{lemh1est} and \ref{lemtest}) in hand, we can now prove the global well-posedness to the Boussinesq system (\ref{maineq}), with initial data $(u_0,\theta_0)\in H^1(\mathbb R^2)$. Specifically, we have the following:

\begin{theorem}\label{welh1}
  For any initial data $(u_0,\theta_0)\in H^1(\mathbb R^2)$, there is a unique global solution $(u,\theta)$ to system (\ref{maineq}), subject to (\ref{ic}), satisfying
  \begin{eqnarray*}
  &&(u,\theta)\in L^\infty(0,T; H^1(\mathbb R^2))\cap C([0,T]; L^2(\mathbb R^2)),\\
  &&(\partial_tu,\partial_t\theta)\in L^2(0,T; L^2(\mathbb R^2)),\quad (\partial_yu,\partial_y\theta)\in L^2(0,T; H^1(\mathbb R^2)),
  \end{eqnarray*}
  for any $T\in(0,\infty)$.
\end{theorem}

\begin{proof}
  The uniqueness is a direct consequence of the next proposition, Proposition \ref{propuniq} (note that the solutions established in this theorem have stronger regularities than those required in Proposition \ref{propuniq}), below, so we only need to prove the existence. Take a sequence $\{(u_{0n},\theta_{0n})\}\in H^2(\mathbb R^2)$, such that
  $$
  (u_{0n},\theta_{0n})\rightarrow(u_0, \theta_0),\quad\mbox{in }H^1(\mathbb R^2).
  $$
  By Proposition \ref{lemglo}, for each $n$, there is a unique global classical solution $(u_n,\theta_n)$ to system (\ref{maineq}), with initial data $(u_{0n},\theta_{0n})$.
  Moreover, by Proposition \ref{lemh1est} and Proposition \ref{lemtest}, we have the following estimate
  \begin{align*}
    \sup_{0\leq s\leq t}\|(u_n(s),\theta_n(s))\|_{H^1}^2+\int_0^t(\|(\partial_tu_n, \partial_t\theta_n)\|_2^2+\|(\partial_yu,\partial_y\theta)\|_{H^1}^2)ds \leq S_1(t)+S_2(t),
  \end{align*}
  for large $n$. Observe that $S_1$ and $S_2$  above, and by Proposition \ref{lemh1est} and Proposition \ref{lemtest},  respectively, depend continuously on $\|(u_{0n},\theta_{0n})\|_{H^1}$. Therefore, we can take them to be slightly larger in the above estimate so that they depend continuously, and solely, on $\|(u_{0},\theta_{0})\|_{H^1}$. Thanks to the above estimate, there is a subsequence, still denoted by $\{(u_n,\theta_n)\}$, and $\{(u,\theta)\}$, such that
  \begin{eqnarray*}
    &(u_n,\theta_n)\overset{*}{\rightharpoonup}(u,\theta)\quad\mbox{ in }L^\infty(0,T; H^1(\mathbb R^2)),\\
    &(\partial_yu_n,\partial_y\theta_n)\rightharpoonup(\partial_yu, \partial_y\theta),\quad\mbox{ in }L^2(0,T; H^1(\mathbb R^2)),\\
    &(\partial_tu_n,\partial_t\theta_n)\rightharpoonup(\partial_tu,\partial \theta),\quad\mbox{ in }L^2(0,T; L^2(\mathbb R^2)),
  \end{eqnarray*}
  for any $T\in(0,\infty)$, fixed, where $\overset{*}{\rightharpoonup}$ and $\rightharpoonup$ denote the weak-* and weak convergences, respectively.
  Moreover, by the Aubin-Lions lemma, for every given positive integer $R$, there is a subsequence, depending on $R$, still denoted by $\{(u_n,\theta_n)\}$, such that
  \begin{eqnarray*}
    &&(u_n,\theta_n)\rightarrow(u,\theta),\quad\mbox{ in }L^2(0,T; L^2(B_R)),\\
    &&(\partial_yu_n,\partial_y\theta_n)\rightarrow(\partial_yu, \partial_y\theta),\quad\mbox{in }L^2(0,T; L^2(B_R)).
  \end{eqnarray*}
  On account of these convergences, one can use a diagonal argument in $n$ and $R$ to show that there is a subsequence, still denoted by $\{(u_n,\theta_n)\}$, such that $(u_n, \theta_n)\rightarrow(u,\theta)$ in $L^2(0,T; L^2(B_\rho))$ and $(\partial_yu_n, \partial_y\theta_n)\rightarrow(\partial_yu,\partial_y\theta)$ in $L^2(0,T; L^2(B_\rho))$, for every $\rho\in(0,\infty)$, in particular that $(u,\theta)$ is a global solution to system (\ref{maineq}), with initial data $(u_0,\theta_0)$. The regularities stated in the theorem follow from the weakly lower semi-continuity of the norms.
\end{proof}

In the next proposition, we relax the assumptions on the regularities of the solutions, in particular on their time derivative, in order to apply the same proposition for proving the uniqueness of the solutions established in Theorem \ref{welh1} as well as those in Proposition \ref{locwel}, below.

\begin{proposition}
  \label{propuniq}
  Given $T\in(0,\infty)$. Let $(u_i,\theta_i)$, $i=1,2$, be two solutions to system (\ref{maineq}) on $\mathbb R^2\times(0,T)$, with initial data $(u_{i 0},\theta_{i 0})\in L^2(\mathbb R^2)$, such that
  \begin{eqnarray*}
  &(u_i,\theta_i)\in L^\infty(0,T; L^2(\mathbb R^2))\cap L^2(0,T; H^1(\mathbb R^2)),\quad(\partial_tu_i,\partial_t\theta_i)\in L^2(0,T; H^{-1}(\mathbb R^2)).
  \end{eqnarray*}
  Then, we have the estimate
  \begin{align*}
    &\sup_{0\leq s\leq t}\|(u_1-u_2,\theta_1-\theta_2)(s)\|_2^2+\int_0^t\|(\partial_y(u_1-u_2), \partial_y (\theta_1-\theta_2))\|_2^2ds\\
    \leq&e^{C\left[t+\sup_{0\leq s\leq t}(\|u_2\|_2^2+\|\theta_2\|_2^2)\int_0^t(\|\partial_xu_2\|_2^2+\|\partial_x \theta_2\|_2^2)ds\right]}\|(u_{10}-u_{20},\theta_{10}-\theta_{20})\|_2^2,
  \end{align*}
  for any $t\in(0,T)$.
\end{proposition}

\begin{proof}
  Note that $(u,\theta)$ satisfies
  \begin{equation}
    \label{eq1}
    \left\{
    \begin{array}{l}
      \partial_tu+(u_1\cdot\nabla)u+(u\cdot\nabla)u_2+\nabla p-\partial_y^2 u=\theta e_2,\\
      \text{div}\,u=0,\\
      \partial_t\theta+u_1\cdot\nabla\theta+u\cdot\nabla\theta_2-\partial_y^2\theta =0.
    \end{array}
    \right.
  \end{equation}
  By the assumption, we have
  \begin{eqnarray*}
  &(u,\theta) \in L^2(0,T; H^1(\mathbb R^2)),\quad(\partial_tu,\partial_t\theta)\in L^2(0,T; H^{-1}(\mathbb R^2)),
  \end{eqnarray*}
  and consequently, by the Lions--Magenes Lemma (cf. \cite{TEMAM}, it follows
  $$
  \langle\partial_tu,u\rangle=\frac12\frac{d}{dt}\|u\|_2^2,\quad\langle \partial_t\theta,\theta\rangle=\frac12\frac{d}{dt}\|\theta\|_2^2,
  $$
  where $\langle\cdot,\cdot\rangle$ denotes the dual action between $H^{-1}(\mathbb R^2)$ and $H^1(\mathbb R^2)$.
  Thanks to these two identities, testing $(\ref{eq1})_1$ and $(\ref{eq1})_3$ by $u$ and $\theta$, respectively, and summing the resultants up, then it follows from integration by parts that
  \begin{align}
    &\frac12\frac{d}{dt}(\|u\|_2^2+\|\theta\|_2^2)+(\|\partial_yu\|_2^2 +\|\partial_y\theta\|_2^2)\nonumber\\
    =&\int_{\mathbb R^2}[-(u\cdot\nabla)u_2\cdot u+\theta u^2-u\cdot\nabla\theta_2\theta]dxdy\nonumber\\
    =&-\int_{\mathbb R^2}(u^2\partial_yu_2\cdot u+u^1\partial_xu_2\cdot u-\theta u^2+u^1\partial_x\theta_2\theta+u^2\partial_y\theta_2\theta)dxdy \nonumber\\
    =&\int_{\mathbb R^2}(\partial_yu^2u_2\cdot u+u^2u_2\cdot\partial_yu- u^1\partial_xu_2\cdot u+\theta u^2 \nonumber\\
    &-u^1\partial_x\theta_2\theta+\partial_yu^2\theta\theta_2+u^2\partial_y \theta\theta_2)dxdy\nonumber\\
    \leq&C\int_{\mathbb R^2}(|u_2||u||\partial_yu|+|u^1||u||\partial_xu_2|+|\theta||u|+|u^1||\theta||\partial_x \theta_2|\nonumber\\
    &+|\partial_yu||\theta||\theta_2|+|u||\partial_y\theta||\theta_2 |)dxdy:=I.\label{eq2}
    \end{align}
    By Lemma \ref{lemhold} and Young inequality, we can estimate the quantity $I$ as follows
    \begin{align*}
    I\leq&C(\|u_2\|_2^{\frac12}\|\partial_xu_2\|_2^{\frac12}\|u\|_2^{\frac12} \|\partial_yu\|_2^{\frac32}+\|u^1\|_2^{\frac12}\|\partial_xu^1\|_2^{\frac12} \|u\|_2^{\frac12}\|\partial_yu\|_2^{\frac12}\|\partial_xu_2\|_2\\
    &+\|u^1\|_2^{\frac12}\|\partial_xu^1\|_2^{\frac12}\|\theta\|_2^{\frac12} \|\partial_y\theta\|_2^{\frac12}\|\partial_x\theta_2\|_2+\|\partial_yu\|_2 \|\theta\|_2^{\frac12}\|\partial_y\theta\|_2^{\frac12}\|\theta_2 \|_2^{\frac12}\|\partial_x\theta_2\|_2^{\frac12}\\
    &+\|u\|_2^{\frac12}\|\partial_yu\|_2^{\frac12}\|\partial_y\theta\|_2\| \theta_2\|_2^{\frac12}\|\partial_x\theta_2\|_2^{\frac12} +\|\theta\|_2\|u\|_2)\\
    \leq&C(\|u_2\|_2^{\frac12}\|\partial_xu_2\|_2^{\frac12}\|u\|_2^{\frac12} \|\partial_yu\|_2^{\frac32}+\|u \|_2 \|\partial_yu \|_2  \|\partial_xu_2\|_2\\
    &+\|u\|_2^{\frac12}\|\partial_yu\|_2^{\frac12}\|\theta\|_2^{\frac12} \|\partial_y\theta\|_2^{\frac12}\|\partial_x\theta_2\|_2+\|\partial_yu\|_2 \|\theta\|_2^{\frac12}\|\partial_y\theta\|_2^{\frac12}\|\theta_2 \|_2^{\frac12}\|\partial_x\theta_2\|_2^{\frac12}\\
    &+\|u\|_2^{\frac12}\|\partial_yu\|_2^{\frac12}\|\partial_y\theta\|_2\| \theta_2\|_2^{\frac12}\|\partial_x\theta_2\|_2^{\frac12} +\|\theta\|_2\|u\|_2)\\
    \leq&\frac12\|(\partial_yu,\partial_y\theta)\|_2^2+
    C(1+\|(u_2,\theta_2)\|_2^2\|(\partial_xu_2,\partial_x\theta_2)\|_2^2) \|(u,\theta)\|_2^2.
  \end{align*}
  Substituting the above estimate into (\ref{eq2}), then by the Gronwall inequality, the conclusion follows.
\end{proof}

\section{Global well-posedness: lower regular initial data}
\label{sec4}
In this section, we prove the global well-posedness of system (\ref{maineq}), with initial data $(u_0^1, u_0^2, \theta_0)$ in space $X=\{f\in L^2(\mathbb R^2)\,|\,\partial_xf\in L^2(\mathbb R^2)\}$. Since we have already established in Theorem \ref{welh1} the global well-posedness result for $H^1$ initial data, and system (\ref{maineq}) has smoothing effect in the $y$ direction, we actually need to show the local well-posedness for initial data in $X$.

We start with the local in time a priori estimates in terms of the $X$-norm of the initial data.

\begin{proposition}
  \label{prop1}
  Let $(u,\theta)$ be a  solution to system (\ref{maineq}), with initial data $(u_0,\theta_0)\in H^2(\mathbb R^2)$. Denote by
  $$
  f_0=1+\|(u_0,\theta_0,\partial_xu_0,\partial_x\theta_0) \|_2^2.
  $$
  Then, there exists a positive absolute constant $C$, such that
\begin{align*}
  \sup_{0\leq s\leq t}\|(u,\theta,\partial_xu,\partial_x\theta)(s)\|_2^2  +\int_0^t\|(\partial_yu,\partial_y\theta,\partial_{xy}^2u,\partial_{xy}^2 \theta)\|_2^2ds
  \leq \sqrt2f_0,
\end{align*}
for any $t\in(0,\frac{1}{4Cf_0^2})$.
\end{proposition}

\begin{proof}
  Multiplying equations $(\ref{maineq})_1$ and $(\ref{maineq})_3$ by $u-\partial_x^2u$ and $\theta-\partial_x^2\theta$, respectively, summing the resultants up, and integrating over $\mathbb R^2$, then it follows from integration by parts, Lemma \ref{lemhold} and the Young inequality that
  \begin{align*}
    &\frac12\frac{d}{dt}\|(u,\theta,\partial_xu,\partial_x\theta)\|_2^2 +\|(\partial_yu,\partial_y\theta,\partial_{xy}^2u,\partial_{xy}^2 \theta)\|_2^2\\
    =&\int_{\mathbb R^2}[u^2\theta-(\partial_xu\cdot\nabla) u\cdot\partial_xu+\partial_x\theta\partial_xu^2-\partial_xu\cdot\nabla\theta \partial_x\theta]dxdy\\
    =&\int_{\mathbb R^2}(u^2\theta+\partial_yu^2\partial_xu\cdot\partial_x u-\partial_xu^2\partial_yu\cdot\partial_xu+\partial_x\theta\partial_x u^2\\
    &\qquad+\partial_yu^2\partial_x\theta\partial_x\theta-\partial_xu^2 \partial_y\theta\partial_x\theta)dxdy\\
    =&\int_{\mathbb R^2}(u^2\theta-2u^2\partial_xu\cdot\partial_{xy}^2 u
    +\partial_{xy}^2u^2u\cdot\partial_xu +\partial_xu^2u\cdot\partial_{xy}^2u+\partial_x\theta\partial_x u^2\\
    &\qquad-2u^2\partial_x\theta\partial_{xy}^2\theta+\partial_{xy}^2u^2 \theta\partial_x\theta+ \partial_xu^2 \theta\partial_{xy}^2\theta)dxdy\\
    \leq&C\int_{\mathbb R^2}[|u||\theta|+(|u|+|\theta|)(|\partial_xu| +|\partial_x\theta|)(|\partial_{xy}^2u|+|\partial_{xy}^2\theta|) +|\partial_x\theta||\partial_xu|]dxdy\\
    \leq&C(\|u\|_2\|\theta\|_2+\|(u,\theta)\|_2^{\frac12}\|(\partial_x u,\partial_x\theta)\|_2 \|(\partial_{xy}^2u,\partial_{xy}^2 \theta)\|_2^{\frac32}+\|\partial_xu\|_2\|\partial_x\theta\|_2)\\
    \leq&\frac12\|(\partial_{xy}^2u,\partial_{xy}^2 \theta)\|_2^2+C(1+\|(u,\theta)\|_2^2+\|(\partial_xu, \partial_x\theta)\|_2^2)^3,
  \end{align*}
  and thus
  \begin{align*}
    \frac{d}{dt}\|(u,\theta,\partial_xu,\partial_x\theta)\|_2^2 +\|(\partial_yu,\partial_y\theta,\partial_{xy}^2u,\partial_{xy}^2 \theta)\|_2^2
    \leq  C(1+\|(u,\theta,\partial_xu, \partial_x\theta)\|_2^2)^3,
  \end{align*}
  for an absolute constant $C$.
  Define
  $$
  f(t)=1+\|(u,\theta,\partial_xu,\partial_x\theta)\|_2^2(t) +\int_0^t\|(\partial_yu,\partial_y\theta,\partial_{xy}^2u,\partial_{xy}^2 \theta)\|_2^2ds,
  $$
  then
  $$
  f'(t)\leq Cf^3(t).
  $$
  Solving the above ordinary differential inequality, one has
  $$
  f(t)\leq\frac{f_0}{\sqrt{1-2Cf_0^2t}}\leq\sqrt2f_0,\quad \mbox{ for }t\in\left(0,\frac{1}{4Cf_0^2}\right),
  $$
  which implies the conclusion.
\end{proof}

Based on the above proposition, we can prove the following local well-posedness result.

\begin{proposition}
  \label{locwel}
  Assume that $(u_0,\theta_0)\in L^2(\mathbb R^2)$, and $(\partial_xu_0,\partial_x\theta_0)\in L^2(\mathbb R^2)$. Then there is a positive time $T_0$ depending only on
  $\|(u_0,\theta_0,\partial_xu_0,\partial_x\theta_0)\|_2^2,$
  such that, system (\ref{maineq}) has a unique solution $(u,\theta)$ on $\mathbb R^2\times(0,T_0)$, with initial data $(u_0, \theta_0)$, satisfying
    \begin{eqnarray*}
    &&(u,\theta)\in C([0,T_0]; L^2(\mathbb R^2))\cap L^2(0,T_0; H^1(\mathbb R^2)),\quad\\
    &&(\partial_xu,\partial_x\theta)\in L^\infty(0,T_0; L^2(\mathbb R^2)),\quad (\partial_{xy}^2u,\partial_{xy}^2\theta)\in L^2(0,T_0; L^2(\mathbb R^2)),\\
    &&(\partial_tu,\partial_t\theta)\in L^2(0,T_0; H^{-1}(\mathbb R^2)).
  \end{eqnarray*}
\end{proposition}

\begin{proof}
  Uniqueness is a direct consequence of Proposition \ref{propuniq}, thus we only need to prove the existence. Take a sequence $\{(u_{0n},\theta_{0n})\}\in H^2(\mathbb R^2)$, such that
  $$
  (u_{0n},\theta_{0n},\partial_xu_{0n},\partial_x\theta_{0n}) \rightarrow(u_0,\theta_0,\partial_xu_0,\partial_x\theta_0),\quad\mbox{ in }L^2(\mathbb R^2).
  $$
  By Lemma \ref{lemglo}, for each $n$, there is a unique global classical solution $(u_n,\theta_n)$ to system (\ref{maineq}), with initial data $(u_{0n},\theta_{0n})$. By Proposition \ref{prop1}, there are two positive constants $T_0$ and $C_0$ depending only on $\|(u_0,\theta_0,\partial_xu_0,\partial_x\theta_0)\|_2^2$, such that
  \begin{align*}
  \sup_{0\leq t\leq T_0}\|(u_n,\theta_n,\partial_xu_n,\partial_x\theta_n)(t)\|_2^2  +\int_0^{T_0}\|(\partial_yu_n,\partial_y\theta_n,\partial_{xy}^2u_n, \partial_{xy}^2 \theta_n)\|_2^2dt
  \leq C_0.
  \end{align*}
  Thanks to this estimate, by the Ladyzhenskaya interpolation inequality for $L^4(\mathbb{R}^2)$ (cf. \cite{Constantin_Foias_1988,TEMAM}), we have
  $$
  \|(u_n,\theta_n)\|_{L^4(\mathbb R^2\times(0,T_0))}^4\leq CC_0.
  $$
Consequently, noticing that $L^2(\mathbb R^2)\hookrightarrow H^{-1}(\mathbb R^2)$, it follows from (\ref{maineq}) that
  \begin{align*}
    \|(\partial_tu_n,\partial_t\theta_n)\|_{L^2(0,T_0; H^{-1}(\mathbb R^2))}^2\leq& C(\|\theta_n\|_{L^2(\mathbb R^2\times(0,T_0))}^2+ \|(u_n,\theta_n)\|_{L^4(\mathbb R^2\times(0,T_0))}^4)\\
    \leq& C(C_0+C_0^2).
  \end{align*}
  On account of these estimates, there is a subsequence, still denoted by $\{(u_n,\theta_n)\}$, and $\{(u,\theta)\}$, such that
  \begin{eqnarray*}
    &(u_n,\theta_n,\partial_xu_n,\partial_x\theta_n)\overset{*}{\rightharpoonup} (u,\theta,\partial_xu,\partial_x\theta),\quad\mbox{ in }L^\infty(0,T_0; L^2(\mathbb R^2)),\\
    &(\partial_yu_n,\partial_y\theta_n,\partial_{xy}^2u_n,\partial_{xy}^2 \theta_n)\rightharpoonup(\partial_yu,\partial_y\theta,\partial_{xy}^2 u,\partial_{xy}^2\theta),\quad\mbox{ in }L^2(0,T_0; L^2(\mathbb R^2)),\\
    &(\partial_t u_n,\partial_t\theta_n)\rightharpoonup(\partial_t u,\partial_t\theta),\quad\mbox{ in }L^2(0,T_0; H^{-1}(\mathbb R^2)).
  \end{eqnarray*}
  Moreover, let $R$ be a given positive integer, by the Aubin--Lions lemma, there is a subsequence, depending on $R$, still denoted by $\{(u_n,\theta_n)\}$, such that
  \begin{eqnarray*}
    &&(u_n,\theta_n)\rightarrow(u,\theta),\quad\mbox{ in }L^2(0,T_0; L^2(B_R)).
  \end{eqnarray*}
  Thanks to these convergences, we can apply the diagonal argument in $n$ and $R$ to show that there is a subsequence of $\{(u_n,\theta_n)\}$ that converges to $(u,\theta)$ in $L^2(0,T_0; L^2(B_\rho))$, for every $\rho\in(0,\infty)$. Furthermore, $(u,\theta)$ is a solution to system (\ref{maineq}) on $\mathbb R^2\times(0,T_0)$, with initial data $(u_0,\theta_0)$. The regularities of $(u,\theta)$ follow from the weakly lower semi-continuity of the norms.
\end{proof}

Now, we are ready to prove the global well-posedness result with initial data in the space $X$, that is  Theorem \ref{theorem}.

\begin{proof}[\textbf{Proof of Theorem \ref{theorem}}]
By Proposition \ref{locwel}, there is a unique local solution $(u,\theta)$ to system (\ref{maineq}) on $\mathbb R^2\times(0,T_0)$, with initial data $(u_0,\theta_0)$, such that
\begin{eqnarray*}
    &&(u,\theta)\in C([0,T_0]; L^2(\mathbb R^2))\cap L^2(0,T_0; H^1(\mathbb R^2)),\quad\\
    &&(\partial_xu,\partial_x\theta)\in L^\infty(0,T_0; L^2(\mathbb R^2)),\quad (\partial_{xy}^2u,\partial_{xy}^2\theta)\in L^2(0,T_0; L^2(\mathbb R^2)),\\
    &&(\partial_tu,\partial_t\theta)\in L^2(0,T_0; H^{-1}(\mathbb R^2)).
  \end{eqnarray*}
For any $t\in(0,T_0)$, recalling that $(u,\theta)\in C([0,t]; L^2(\mathbb R^2))\cap L^2(0,t; H^1(\mathbb R^2))$, there is a time $t_1\in(\frac t2,t)$, such that $(u(t_1),\theta(t_1))\in H^1(\mathbb R^2)$. Now, choosing $t_1$ as the initial time, by Theorem \ref{welh1}, one can extend solution $(u,\theta)$ uniquely to arbitrary finite time $T$, such that
\begin{eqnarray*}
&&(u,\theta)\in L^\infty(t_1,T; H^1(\mathbb R^2))\cap C([t_1,T];L^2(\mathbb R^2)),\\
&&(\partial_tu,\partial_t\theta)\in L^2(t_1,T; L^2(\mathbb R^2)),\quad(\partial_yu,\partial_y\theta)\in L^2(t_1,T; H^1(\mathbb R^2)).
\end{eqnarray*}
Therefore, we have established a global solution $(u,\theta)$ to system (\ref{maineq}), subject to (\ref{ic}), satisfying the regularities in the theorem. This completes the proof.
\end{proof}

\section*{Acknowledgments}
J.L. is thankful to the warm hospitality of the Department of Mathematics, Texas A\&M University, where  part of this work was completed. This work is supported in part by a grant of the ONR, and by the  NSF
grants DMS-1109640 and DMS-1109645.

\end{document}